\documentclass{amsart}
\usepackage{amsmath,amsfonts} 
\usepackage{amsthm} 
\usepackage{amssymb} 
\usepackage[T1]{fontenc} 
\usepackage[utf8]{inputenc}
\usepackage{textcomp}
\usepackage[dvips]{graphicx}
\usepackage{verbatim}
\usepackage{mathrsfs}
\usepackage[english]{babel}
\usepackage[babel]{csquotes}
\usepackage{xy}
\usepackage[final]{microtype}
\xyoption{all}
\usepackage[style=alphabetic,hyperref,backend=bibtex]{biblatex}
\bibliography{bibliography}

\theoremstyle{definition}
\newtheorem{definition}{Definition}[section]
\newtheorem{generalization}{Theorem}
\newtheorem*{counterexample}{Counterexample}
\theoremstyle{plain}
\newtheorem{theorem}[definition]{Theorem}
\newtheorem{proposition}[definition]{Proposition}
\newtheorem{lemma}[definition]{Lemma}
\newtheorem{corollary}[definition]{Corollary}

\theoremstyle{remark}
\newtheorem{remark}[definition]{Remark}
\newtheorem*{claim}{Claim}
\newtheorem*{acknowledgments}{Acknowledgments}
\newtheorem*{notation}{Notation}

\DeclareMathOperator{\Exp}{Exp}

\DeclareMathOperator{\Spec}{Spec}

\DeclareMathOperator{\GL}{GL}

\DeclareMathOperator{\trdeg}{trdeg}

\DeclareMathOperator{\Strat}{Strat}
\DeclareMathOperator{\Aut}{Aut}
\DeclareMathOperator{\End}{End}
\DeclareMathOperator{\Vt}{Vecf}
\DeclareMathOperator{\coh}{H}

\newcommand{\aut}{\underline{\Aut}^\otimes}
\newcommand{\E}{\mathbb{E}}
\newcommand{\F}{\mathbb{F}}

\newcommand{\Z}{\mathbb{Z}}
\newcommand{\I}{\mathbb{I}}
\newcommand{\Ha}{\mathbb{H}}
\newcommand{\ebar}{{\bar\eta}}
\newcommand{\Oh}{\mathscr{O}}
\newcommand{\M}{\mathcal{M}}
\newcommand{\N}{\mathbb{N}}

\newcommand{\D}{\mathscr{D}}

\newcommand{\de}{\partial}

\newcommand{\A}{\mathbb{A}}
\newcommand{\U}{\mathcal{U}}
\newcommand{\bin}[2]{\genfrac{(}{)}{0pt}{}{#1}{#2}}
\newcommand{\onto}{\twoheadrightarrow}

\makeatletter
\def\blfootnote{\gdef\@thefnmark{}\@footnotetext}
\makeatother

\begin{document}
\title[The variation of the monodromy group in families]{The variation of the monodromy group in families of stratified bundles \\ in positive characteristic}
\author{Giulia Battiston}
\date{\today}
\address{}

\email{gbattiston@mi.fu-berlin.de}

\begin{abstract}In this article we study smooth families of stratified bundles in positive characteristic and the variation of their monodromy group. Our aim is, in particular, to strengthen the weak form of the positive equicharacteristic $p$-curvature conjecture stated and proved by Esnault and Langer in \cite{EL:pcur}. The main result is that if the ground field is uncountable then the strong form holds. In the case where the ground field is countable we provide positive and negative answers to possible generalizations.
\end{abstract}
\maketitle
\blfootnote{\subjclass{2010 MSC Primary: 14D05, 14H30, 13N10}.\\ This work was supported by SFB/TR 45 ``Periods, moduli spaces and arithmetic of algebraic varieties'' and GRK 1800 ``Moduli and Automorphic Forms'' of the DFG}
\section{Introduction}
Let $(E,\nabla)$ be a vector bundle endowed with a flat connection on a smooth complex variety $X$. Then, there exists a smooth scheme $S$ over (some open subscheme of) $\Spec\mathbb{Z}$ such that $(E,\nabla)=(E_S,\nabla_S)\otimes_S\mathbb{C}$ and $X=X_S\otimes_S\mathbb{C}$ with $X_S$ smooth over $S$ and $(E_S,\nabla_S)$ flat connection on $X_S$ relative to $S$. The $p$-curvature conjecture of Grothendieck and Katz predicts (see \cite[Conj.~3.3.3]{An:pcur}) that if for all closed points $s$ of a dense open subscheme $\tilde{S}\subset S$ we have that $E_S\times_Ss$ is spanned by its horizontal sections, then  $(E,\nabla)$ must be trivialized by an \'etale finite cover of $X$.

An analogue problem can be studied in equicharacteristic zero, and in fact it reduces the $p$-curvature conjecture to the number field case. Y.~Andr\'e in \cite[Prop.~7.1.1]{An:pcur} and E.~Hrushovski in \cite[116]{Hru} stated and proved the following equicharacteristic zero version of the $p$-curvature conjecture:  let $X\to S$ be a smooth morphism of varieties over a field $K$ of characteristic zero; let $(E,\nabla)$ be a flat connection on $X$ relative to $S$ such that, for every  closed point $s$ in a dense open $\tilde{S}\subset S$, the flat connection $(E,\nabla)\times_Ss$ is trivialized by a finite \'etale cover. Then, there exists a finite \'etale cover of the generic geometric fiber over $\ebar$ that trivializes $(E,\nabla)\times_S\ebar$, where $\ebar$ is a geometric generic point of $S$.

The theorem of Andr\'e and Hrushovsky translates naturally in positive characteristic, providing a positive equicharacteristic analogue to the $p$-curvature conjecture. Here, the role of relative flat connections is played by relative stratified bundles. A \emph{stratified bundle on $X$ relative to $S$} is a vector bundle of finite rank with an action of the ring of differential operators $\D_{X/S}$ on $X$ relative to $S$. 

In \cite[Cor.~4.3, Rmk.~5.4.1]{EL:pcur} H.~Esnault and A.~Langer  proved, using an example of Y.~Laszlo (see \cite{Ls}), that there exists $X\to S$ a projective smooth morphism  of varieties over $\bar{\F}_2$ and a stratified bundle over $X$ relative to $S$ which is trivialized by a finite \'etale cover on every closed fiber but not on the geometric generic one. In particular, this provides a counterexample to the positive equicharacteristic version of Andr\'e and Hrushovsky's theorem.

Nevertheless, they were able to prove what they call a weak form of the theorem (see \cite[Thm.~7.2]{EL:pcur}): let $X\to S$ be a \emph{projective} smooth morphism and let $\E=(E,\nabla)$ be a stratified bundle on $X$ relative to $S$ such that, for all closed points of a dense open $\tilde{S}\subset S$, the stratified bundle $\E\times_Ss$ is trivialized by a finite \'etale cover \emph{of order prime to $p$}. Then, if $K\neq\bar{\F}_p$, there exists a finite \'etale cover \emph{of order prime to $p$} of the generic geometric fiber that trivializes $\E\times_S\ebar$. In case $K=\bar{\F}_p$, there exists a finite \'etale cover \emph{of order prime to $p$} such that the pullback of $\E\times_S\ebar$ is a direct sum of stratified line bundles.

\

In this article we answer two natural questions aiming at generalizing this last theorem: the first one is whether we can relax the assumption of coprimality to $p$ of the order of the trivializing covers of the $\E\times_Ss$, while keeping the assumption that $X$ is proper over $S$. The second one is if we can drop this last assumption as well. Our main result, in particular, is that if $K$ is uncountable then the positive equicharacteristic version of Andr\'e and Hrushovsky's theorem holds.

\

Let assume from now on that $K$ has positive characteristic $p$. Bearing in mind the counterexample of Esnault and Langer (\cite[Cor.~4.3]{EL:pcur}) we cannot hope in general to completely eliminate the assumption of coprimality to $p$ of the order of the trivializing covers of the $\E\times_Ss$. Still, we can answer positively the first question proving that it suffices to impose to the power of $p$ dividing the order of such trivializing covers to be bounded:

\begin{generalization}[See Theorem~\ref{eldiv}]\label{uno} Let $K$ be an algebraically closed field, $X\to S$  a smooth proper morphism of $K$-varieties and $\E=(E,\nabla)$ a stratified bundle on $X$ relative to $S$. Assume that for every closed point $s$ in a dense open $\tilde{S}\subset S$ the stratified bundle $\E_s=\E\times_Ss$ is trivialized by a finite \'etale cover \emph{whose order is not divisible by $p^N$ for some fixed $N\geq 0$}. Then, if $K\neq\bar{\F}_p$, there exists a finite \'etale cover of the generic geometric fiber that trivializes $\E_\ebar=\E\times_S\ebar$. In case $K=\bar{\F}_p$, there exists a finite \'etale cover such that the pullback of $\E_\ebar$ is the direct sum of stratified line bundles.
\end{generalization}

The second question has a more involved answer. The assumption on $X$ being proper over $S$ is more delicate to eliminate; the order of the trivializing covers does not play any role while the cardinality of $K$ becomes the main obstruction:

\begin{counterexample}[See Proposition~\ref{ex}] If $K$ is a countable field, then there exists a stratified bundle on $\A_K^2$ relative to $\A_K^1$ which is trivial on every closed fiber but is not trivialized by any finite \'etale cover on the generic geometric fiber.
\end{counterexample}

On the other hand the main result of this article is that in case $K$ is uncountable the strong version of the theorem holds, namely:

\begin{generalization}[See Theorem~\ref{main}]\label{due} Let $K$ be an \emph{uncountable} algebraically closed field, $X\to S$  a smooth morphism of $K$-varieties and $\E=(E,\nabla)$ a stratified bundle on $X$ relative to $S$ such that, for every  closed point $s$ in a dense open $\tilde{S}\subset S$, the stratified bundle $\E_s=\E\times_Ss$ is trivialized by a finite \'etale cover. Then, there exists a finite \'etale cover of the generic geometric fiber that trivializes $\E_\ebar=\E\times_S\ebar$.
\end{generalization}

In the case where $K$ is countable and $X$ is not proper over $S$ there is still room for improvement, using the theory of regular singular stratified bundles (introduced in \cite{Gie:flat}). Roughly speaking, a stratified bundle is regular singular if it has only mild (that is logarithmic) singularities along the divisor at infinity. In characteristic zero there is a parallel notion of regular singular flat connections, and one of the first steps in the proof of Andr\'e's theorem is to show that if a relative flat connection $(E,\nabla)$ on $X$ over $S$ is regular singular on the fiber over all closed points of a dense subset of $S$ then it is regular singular on the generic fiber (see \cite[Lemma~8.1.1]{An:pcur}). In positive characteristic this is no longer true, as our counterexample shows. The converse still holds (see the proof of Lemma \ref{torsion}): if $X$ admits a good compactification over $S$ and  $\E=(E,\nabla)$ is a stratified bundle on $X$ relative to $S$ such that $\E_\ebar$ is regular singular then for every closed point $s$ of some dense open $\tilde{S}\subset S$ the stratified bundle $\E_s$ is regular singular as well. Moreover, assuming $\E_\ebar$ to be regular singular we obtain the same results than in the proper case:

\begin{generalization}[See Theorem~\ref{smoothrs}]\label{tre} Let $K$ be an algebraically closed field of any cardinality and $X\to S$  a smooth morphism of $K$-varieties. Let $\E=(E,\nabla)$ be a stratified bundle on $X$ relative to $S$ such that, for every  closed point $s$ in a dense open $\tilde{S}\subset S$, the stratified bundle $\E_s=\E\times_Ss$ is trivialized by a finite \'etale cover whose order is not divisible by $p^N$ for some fixed $N\geq 0$. Assume moreover that $\E_\ebar=\E\times_S\ebar$ is \emph{regular singular}. Then, if  $K\neq\bar{\F}_p$, there exists a finite \'etale cover of the generic geometric fiber that trivializes $\E_\ebar$. In case $K=\bar{\F}_p$, there exists a finite \'etale cover such that the pullback of $\E_\ebar$ is the direct sum of stratified line bundles. 
\end{generalization}

The proofs of these generalizations are of two different kinds. The ones of Theorem~\ref{uno} and Theorem~\ref{tre} rely on a reduction to Esnault and Langer's result (\cite[Thm.~7.2]{EL:pcur}). Theorem~\ref{tre} is reduced to Theorem~\ref{uno} using the theory of exponents for regular singular stratified bundles (\cite{Gie:flat},\cite{Kin:rs}) and an adaptation of Kawamata coverings (\cite[Thm.~17]{kaw}) to positive characteristic (see Theorem~\ref{kawamata}). Theorem~\ref{uno} is then proved by reduction to \cite[Thm.~7.2]{EL:pcur} constructing a suitable finite \'etale cover of $X$ that kills the $p$-powers in the orders of the trivializing covers on the closed fibers.

The proof of Theorem~\ref{due} is of another flavour: it relies on the invariance of the Tannakian monodromy group under algebraically closed extension of fields for finite stratified bundles (Lemma~\ref{basechange}), and on the easy but fundamental fact that a (relative) stratified bundle is defined by countably many data.

\begin{acknowledgments} The results contained in this article are part of my PhD work under the supervision of H\'el\`ene Esnault. I would like to thank her for introducing me to the subject, and for her great patience and support. Moreover, I would like to thank Lars Kindler for many useful and pleasant discussions, in particular regarding Lemma \ref{finite}.
\end{acknowledgments}

\begin{notation}
If $S$ is an integral scheme  $k(S)$ will denote its field of fractions, $\eta$ its generic point and $\ebar$ a geometric generic point given by the choice of an algebraically closure $\overline{k(S)}$ of $k(S)$.
If $K$ is a field a \emph{variety} over $K$ is a separated integral scheme of finite type over $K$.
\end{notation}

\section{The category of stratified bundles}
Throughout the whole article \(K\) will denote an algebraically closed field of positive characteristic \(p\) and $u:X\to S$ a smooth morphism of varieties over $K$, of relative dimension $d$. Let $\D_{X/S}$ be the quasi-coherent $\Oh_X$-module of relative differential operators as defined in \cite[§16]{EGA}; recall that if $\U$ is an open subscheme of $X$ admitting global coordinates $x_1,\dotsc,x_d$ relative to $S$, then for every $k\in\N$ there are $\Oh_S$-linear maps $\de_{x_i}^{(k)}:\Oh_{\U}\to\Oh_{\U}$ given by
\[\de_{x_i}^{(k)}(x_j^h)=\delta_{ij}\bin{h}{k}(x_j^{h-k})\]
where $\delta_{ij}$ is the Kronecker delta. These maps are differential operators of order $k$ and generate locally the ring of differential operators:
\[\D_{X/S|_\U}=\Oh_{\U}\big[\de_{x_i}^{(k)}\mid i\in\{1,\dotsc,d\},k\in\N_{>0}\big].\]
For higher differential operators we have an extension of the Leibniz rule, namely if $f,g\in\Oh_{\U}$ then
\begin{equation}\label{product}\de_{x_i}^{(k)}(fg)=\sum_{\substack{a+b=k\\a,b\geq 0}}\de_{x_i}^{(a)}(f)\de_{x_i}^{(b)}(g).\end{equation}

\begin{definition}A \emph{stratified bundle $\E$ (relative to S)}  is a $\Oh_X$-locally free module $E$ of finite rank $r$ endowed with a $\D_{X/S}$-action extending the $\Oh_X$-module structure via the inclusion $\Oh_X\subset\D_{X/S}$. A \emph{morphism} of stratified bundles is a morphism of $\D_{X/S}$-modules. We denote by $\Strat(X/S)$ the category of stratified bundles on $X$ relative to $S$; if $S=\Spec K$ we use the notation $\Strat(X/K)$ for $\Strat(X/\Spec K)$. The structure sheaf $\Oh_X$ together with the natural $\D_{X/S}$-action is denoted by $\I_{X/S}$; if $S=\Spec K$ we use the notation $\I_{X/K}$ for $\I_{X/\Spec K}$.   A stratified bundle is \emph{trivial} if it is isomorphic to $\I_{X/S}^{\oplus r}$ for some $r\in\N$.
\end{definition}
If $h:Y\to X$ is a morphism of smooth $S$-varieties  then the pullback along $h$ induces a functor $h^*:\Strat(X/S)\to\Strat(Y/S)$ and if $h$ is finite and \'etale then the pushforward along $h$ induces a functor $h_*:\Strat(Y/S)\to\Strat(X/S)$.
For $\E,\F\in\Strat(X/S)$ we can construct the dual $\E^\lor$, the tensor product $\E\otimes\F$ and the direct sum $\E\oplus\F$, all of which are objects of $\Strat(X/S)$.

\section{The monodromy group}\label{monodromy}
If $X$ is a smooth connected $K$-variety, $\Strat(X/K)$ is an abelian tensor rigid $K$-linear category and the choice of a rational point $x\in X(K)$ defines a fiber functor to the category of finite dimensional $K$-vector spaces by:
\[\begin{split}
\omega_x:\Strat(X/S)&\to \Vt_K \\
\E&\mapsto E_x
\end{split}\]
where $E$ is the vector bundle underlying $\E$ (\cite[§VI.1]{Sa:ct}). Hence $(\Strat(X/S),\omega_x)$ is a neutral Tannakian category and by Tannakian duality (\cite[Thm.~2.11]{DM}) there exists an affine group scheme $\pi_1^{\Strat}(X,x)\doteq\pi(\Strat(X/S),\omega_x)=\aut_K(\omega_x)$ over $K$ such that $\Strat(X/K)$ is equivalent via $\omega_x$ to the category of finite dimensional representations of $\pi_1^{\Strat}(X,x)$ over $K$. For every $\E\in\Strat(X/K)$ we denote by $\langle\E\rangle_\otimes\subset\Strat(X/K)$ the full Tannakian subcategory spanned by $\E$ with fiber functor $\omega_x$ defined as above. The affine group scheme  $\pi(\E,x)\doteq\pi(\langle\E\rangle_\otimes,\omega_x)$ is called the \emph{monodromy group} of $\E$. If $\U\subset X$ is an open dense subscheme of $X$ then by \cite[Lemma~2.5(a)]{Kin:rs} the restriction functor $\rho_{\U}:\langle\E\rangle_\otimes\to\langle\E_{|\U}\rangle_\otimes$ is an equivalence; hence, in particular, the monodromy group of $\E$ is invariant under restriction to dense open subschemes. Moreover, as $K$ is algebraically closed, the monodromy group does not depend on the choice of $x$, up to non-unique isomorphism (this can be deduced from \cite[Thm.~3.2]{DM}). We will hence sometimes use the notation $\pi(\E)$ instead of $\pi(\E,x)$. 

\begin{definition} We say that $\E\in\Strat(X/K)$ is \emph{finite} if its monodromy group is finite. By what we have just remarked, this property is independent of the choice of $x$. We say that $\E$ is \emph{isotrivial} if it is \'etale trivializable; that is, there exists $h:Y\to X$ finite \'etale cover such that $h^*\E$ is trivial in $\Strat(Y/K)$.
\end{definition}
These two properties are equivalent:

\begin{lemma}\label{finite} For a stratified bundle $\E\in\Strat(X/K)$ the following are equivalent:
\begin{itemize}
\item[i)] $\E$ is isotrivial;
\item[ii)]$\E$ is finite.
\end{itemize}
Moreover, if $\E$ is finite, then there exists an \'etale $\pi(\E,x)$-torsor $h_{\E,x}:Y_{\E,x}\to X$, called the \emph{Picard--Vessiot torsor} of $\E$ such that, for any $\E'\in\Strat(X/K)$, the pullback $h_{\E,x}^*\E'$ is trivial if and only if $\E'\in\langle\E\rangle_\otimes$. 

Finally, for a finite \'etale cover $h:Y\to X$ such that $h^*\E$ is trivial, the following conditions are equivalent:
\begin{itemize}
\item[i)]$h:Y\to X$ is the Picard--Vessiot torsor for $\E$;
\item[ii)] every finite \'etale cover trivializing $\E$ factors (non-uniquely) through $h:Y\to X$;
\item[iii)] $h:Y\to X$ is Galois and $\langle\E\rangle_\otimes=\langle h_*\I_{Y/K}\rangle_\otimes$;
\item[iv)] $h:Y\to X$ is Galois of Galois group $\pi(\E,x)(K)$.
\end{itemize}
\end{lemma}
\begin{proof}
The first part of the lemma is \cite[Lemma~1.1]{EL:pcur}. As for the second part, first notice that point (b) and (f) of \cite[Prop.~2.15]{Kin:rs}, together with \cite[Cor.~2.16]{Kin:rs} imply that if $h:Y\to X$ is a finite \'etale cover trivializing $\E$ then $\langle\E\rangle_\otimes\subset\langle h_*\I_{Y/K}\rangle_\otimes$ and that  $\langle\E\rangle_\otimes=\langle h_{\E,x*}\I_{Y_{\E,x}/K}\rangle_\otimes$. Moreover, if $h:Y\to X$ is Galois of Galois group $G$, then $\pi(h_*\I_{Y/K},x)$ is the finite constant group $G$ and if $\tilde{h}:\tilde{Y}\to X$ is an \'etale cover factoring through $h$ then $\langle h_*\I_{Y/K}\rangle_\otimes\subset\langle \tilde{h}_*\I_{\tilde{Y}/K}\rangle_\otimes$. We are now ready to prove the rest of the lemma.
\begin{itemize}
\item[(i)$\Rightarrow$(ii)]Because  $\langle\E\rangle_\otimes=\langle h_{\E,x*}\I_{Y_{\E,x}/K}\rangle_\otimes$, a cover $\tilde{h}:\tilde{Y}\to X$ trivializes $\E$ if and only if it trivializes $h_{\E,x*}\I_{Y_{\E,x}/K}$. Let $Z=\tilde{Y}\times_X Y_{\E,x}$, and let $p_1$ and $p_2$ be the projections on the first and second factor. Then by flat base change (notice that the flat base change morphism is compatible with the $\D_{\tilde{Y}/K}$-action) there is an isomorphism of $\D_{\tilde{Y}/K}$-modules $\tilde{h}^*h_{\E,x*}\I_{Y_{\E,x}/K}\simeq p_{1*}\I_{Z/K}$; hence, the  latter is also a trivial stratified bundle. This, together with \cite[Cor.~2.17]{Kin:rs}, implies that $p_1:Z\to \tilde{Y}$ is a trivial covering. In particular, it admits a section $s$; hence, $\tilde{h}=s\circ p_2\circ h_{\E,x}$ and $\tilde{h}$ factors through $h$.

\item[(ii)\(\Rightarrow\)(iii)] Because $h$ trivializes $\E$, we have the inclusion $\langle\E\rangle_\otimes\subset\langle h_*\I_{Y/K}\rangle_\otimes$. On the other side, by assumption, $h_{\E,x}:Y_{\E,x}\to X$ factors through $h:Y\to X$; hence, $\langle h_*\I_{Y/K}\rangle_\otimes\subset \langle h_{\E,x*}\I_{Y_{\E,x}/K}\rangle_\otimes=\langle\E\rangle_\otimes$.

\item[(iii)$\Rightarrow$(iv)] As $\langle\E\rangle_\otimes=\langle h_*\I_{Y/K}\rangle_\otimes$, then we have the equality $\pi(\E,x)=\pi(h_*\I_{Y/K},x)$ and as $h:Y\to X$ is Galois, then its Galois group is $\pi(h_*\I_{Y/K})(K)=\pi(\E,x)(K)$.

\item[(iv)$\Rightarrow$ (i)] By what we already proved there must be a factorization $f:Y\to Y_{\E,x}$ such that $h=h_{\E,x}\circ f$. Hence, if $G$ is the Galois group of $h:Y\to X$ then $h_\E:Y_\E\to X$ corresponds to a normal subgroup $H$ of $G$. But by assumption $G=\pi(\E,x)(K)=H$; hence, $h=h_\E$. \qedhere
\end{itemize}
\end{proof}

\begin{corollary}\label{minimal}If $\E\in \Strat(X/K)$ is finite then the set of finite \'etale covers of $X$ trivializing $\E$ has a minimal element which is Galois of Galois group $\pi(\E,x)(K)$.
\end{corollary}

By \cite[Cor.~12]{dS:fun} for every $\E\in\Strat(X/K)$ the group scheme $\pi(\E,x)$ is smooth (which is equivalent to being reduced). In particular, as $K$ is algebraically closed, if $\E$ is finite we are allowed  to identify the abstract group $\pi(\E,x)(K)$ and the algebraic group $\pi(\E,x)$.
Given a finite stratified bundle $\E\in\Strat(X/K)$ it is straightforward to see that for every $L\supset K$ algebraically closed field extension $\E_L=\E\otimes_KL\in\Strat(X_L/L)$ is finite as well. In fact, the following stronger statement holds:

\begin{lemma}\label{basechange} Let $\E\in\Strat(X/K)$ and let $L\supset K$ be an algebraically closed field extension such that $\E_L$ is finite. Then for every $L'\supset K$ algebraically closed such that $\trdeg_KL'\geq\trdeg_KL$ (or, if both are infinite, such that there exists an immersion $L\hookrightarrow L'$ which is the identity on $K$) we have that $\E_{L'}$ is finite. Moreover for any $x\in X(K)$  
\[\pi(\E_L,x)(L)\simeq\pi(\E_{L'},x)(L'),\] 
 where we consider $x\in X_L(L)$ via $K\subset L$ and similarly for $L'$.
\end{lemma}
\begin{proof}
Let $L$ and $L'$ as in the hypothesis, then we can construct an immersion $L\hookrightarrow L'$ which is the identity on $K$, just by sending any transcendence basis of $L$ to a algebraically independent set in $L'$ over $K$ and using the fact that $L'$ is algebraically closed to see that this extends to an immersion $L\hookrightarrow L'$. Hence, we have reduced the problem to proving that if $\E$ is finite and $L\supset K$ is an algebraically closed field extension then $\E_L$ is finite and has the same monodromy group of $\E$ as abstract groups. In order to do so we need first to establish a result on Galois covers:
\begin{claim}
Let $h:Y\to X$ be a Galois cover of Galois group $G$ and let $h_L:Y_L\to X_L$ the extension of scalars of $h:Y\to X$ to $L$, then $h_L$ is a Galois cover of Galois group $G$.
\end{claim}
\begin{proof}
Certainly $h_L:Y_L\to X_L$ is a finite \'etale morphism as these properties are stable under base change. We are left to check that (i) $Y_L$ is connected, (ii) $\Aut(Y_L/X_L)$ acts transitively on the fiber over some geometric point of $X_L$ and finally (iii) $\Aut(Y_L/X_L)\simeq \Aut(Y/X)$.
\begin{itemize}
\item[i)] As $K$ is algebraically closed (hence, in particular, separably closed) $Y$ is connected if and only if $Y_L$ is connected for any field extension $L\supset K$. In particular, $Y_L$ is connected.
\item[ii)] Let $x\in X_L(L)$ be any closed (in particular, geometric) point of $X_L$, then the composition $\bar{x}:\Spec(L)\to X_L\to X$ is a geometric point for $X$. As $h:Y\to X$ is Galois, $\Aut(Y/X)$ acts transitively on $Y_{\bar{x}}=Y\times_X\bar{x}=Y_L\times_{X_L}x=Y_{L,x}$. Now, the action of $\Aut(Y/X)$ on $Y_{L,x}$ factors through $\Aut(Y_L/X_L)$ via the inclusion $\Aut(Y/X)\subset\Aut(Y_L/X_L)$ defined by $\phi\mapsto \phi_L$. Hence, the action of $\Aut(Y_L/X_L)$ on $Y_{L,x}$ is transitive as well; therefore, $h_L:Y_L\to X_L$ is Galois.
\item[iii)]As $Y_{\bar{x}}=Y_{L,x}$ and both $h$ and $h_L$ are Galois, it follows that the order of their Galois group is the same, as it is the cardinality of the respective geometric fibers over $\bar{x}$ and over $x$. Moreover we have a natural inclusion $\Aut(Y/X)\subset\Aut(Y_L/X_L)$ so as they have the same cardinality they must be equal; hence, $\Aut(Y_L/X_L)=G$.\qedhere
\end{itemize}
\end{proof}
Until the end of the proof let us denote by $h_{\E,x}:Y\to X$ the Picard--Vessiot torsor of $\E$ (see Lemma \ref{finite}), then $h_{\E,x}\otimes_KL:Y_L\to X_L$ is a Galois cover trivializing $\E_L$ which is then finite, by Lemma \ref{finite}. Recall that  $\langle (h_{\E,x})_*\I_{Y/K}\rangle_\otimes=\langle\E\rangle_\otimes$. But then in particular, $\langle (h_{\E,x})_*\I_{Y/K}\otimes_KL\rangle_\otimes=\langle\E_L\rangle_\otimes$ and  as $(h_{\E,x})_*\I_{Y/K}\otimes_KL=(h_{\E,x}\otimes_KL)_*\I_{Y_L/L}$, it follows that $\langle(h_{\E,x}\otimes_KL)_*Y_L \rangle_\otimes=\langle\E_L\rangle_\otimes$. Hence, by Lemma \ref{finite}, we have that $h_{\E,x}\otimes_KL:Y_L\to X_L$ is the minimal trivializing cover for $\E_L$. Now, the Galois group of $h_{\E,x}$ is the same as that of $h_{\E,x}\otimes_KL$ by the previous claim; hence, again by Lemma \ref{finite}, we have that $\pi(\E_L,x)(L)=\pi(\E,x)(K)$.
\end{proof}

Finite stratified bundle have an additional propriety that will turn out to be very useful to prove that some stratified bundle cannot be isotrivial:

\begin{lemma}\label{finito}
Let $\E\in\Strat(X/K)$ be a finite stratified bundle. Then there exists a subfield $K'\subset K$ of finite type over $\F_p$ over which $X$ and $\E$ are defined; that is, there exists $X'$ smooth variety over $K'$ and $\E'\in\Strat(X'/K')$ such that $X=X'\times_{\Spec K'}\Spec K$ and $\E=\E'\otimes_{K'}K$.
\end{lemma}
\begin{proof}
Let $h_{\E,x}:Y_{\E,x}\to X$ be the Picard--Vessiot torsor of $\E$ (see Lemma \ref{finite}), and let $\Ha=(h_{\E,x})_*\I_{Y_{\E,x}/K}$, then (see Lemma \ref{finite}) $\E\in\langle \Ha\rangle_\otimes$.
Certainly there exists $K''$ of finite type over $\F_p$ on which $h_{\E,x}:Y_{\E,x}\to X$ is defined; hence, $\Ha$ is also defined over $K''$. Notice that $\E$ is a subquotient of $\mathbb{P}$ where $\mathbb{P}\in\Z[\Ha,\Ha^\lor]$ (see e.g. \cite[def.~2.4]{Kin:rs}); that is, $\E\simeq \tilde{\mathbb{P}}/\bar{\mathbb{P}}$ with $\bar{\mathbb{P}}\subset\tilde{\mathbb{P}}\subset\mathbb{P}$. It is clear that $\tilde{\mathbb{P}}$ and $\bar{\mathbb{P}}$ are defined over some extension $K'$ of finite type of $K''$ (thus  over $\F_p$). Therefore, so does $\E\simeq \tilde{\mathbb{P}}/\bar{\mathbb{P}}$.
\end{proof}

In particular, we have the following:
\begin{corollary}
Let $\E\in\Strat(X/K)$ with $K$ algebraically closed, such that for some algebraically closed field extension $L\supset K$ the stratified bundle $\E_L$ is finite. Assume that $K$ has infinite transcendence degree over $\F_p$, then $\E$ is finite as well.
\end{corollary}
\begin{proof}
As $\E_L$ is finite then by Lemma~\ref{finito} there exists $K'\subset L$ algebraically closed of finite transcendence degree over $\F_p$ over which $\E_L$ and  is defined. But then $\E$ is defined over $K'$ as well, hence we can assume that $K'\subset K$. Let $h_{\E_L}:Y\to X_L$ be the Picard--Vessiot torsor of $\E_L$, then it is defined over some algebraically closed field $K''$ of finite transcendence degree over $K'$. In particular $\E_{K''}$ is finite and as $\infty=\trdeg_{K'}K\geq\trdeg_{K'}K''$ by Lemma~\ref{finite} $\E=\E_{K'}\otimes_{K'}K$ is finite.
\end{proof}

\section{Families of finite stratified bundles} We can view a relative stratified bundle $\E\in\Strat(X/S)$ as a family of stratified bundles parametrized by the points of $S$. In particular, for every $s\in S(K)$, let $\E_s\in\Strat(X_s/k(s))$ denote the restriction of $\E$ on $X_s$ and  $\E_\ebar\in\Strat(X_\ebar/k(\ebar))$ its restriction on the geometric generic fiber given by a choice of an algebraic closure $\overline{k(S)}$ of $k(S)$. It is natural then to ask how the property of being isotrivial behaves in families: the main question we want to study is whether it is true that if $\E_s$ is finite for every $s\in S(K)$ then so is $\E_\eta$.
Andr\'e proved in \cite[Prop.~7.1.1]{An:pcur} that the analogous result in characteristic zero holds. In positive characteristic, following an idea of Laszlo, in \cite[Cor.~4.3, Rmk.~5.4.1]{EL:pcur} the authors proved that there exists $X\to S$ a projective smooth morphism of varieties over $\bar{\F}_2$ and a stratified bundle on $X$ relative to $S$ which is finite on every closed fiber but not on the geometric generic one. Nevertheless, assuming $X$ to be projective over $S$ and imposing a coprimality to $p$ condition on the order of the monodromy group on the closed fibers, they proved the following:

\begin{theorem}\emph{\cite[Thm.~7.2]{EL:pcur}}\label{el} Let $X\to S$ be a  smooth \emph{projective} morphism of $K$-varieties with geometrically connected fibers and let $\E\in\Strat(X/S)$. Assume that there exists a dense subset $\tilde{S}\subset S(K)$ such that, for every $s\in \tilde{S}$, the stratified bundle $\E_s$ has finite monodromy \emph{of order prime to $p$}. Then
\begin{itemize}
\item[i)] there exists $f_\ebar:Y_\ebar\to X_\ebar$ a finite \'etale cover \emph{of order prime to $p$} such that $f^*\E_\ebar$ decomposes as direct sum of stratified line bundles;
\item[ii)] if $K\neq \bar{\F}_p$ then $\E_\ebar$ is trivialized by a finite \'etale cover \emph{of order prime to $p$}.
\end{itemize} 
\end{theorem}

Note that the cover of order prime to $p$ in the second point of the theorem factors through the Picard--Vessiot torsor of $\E_\ebar$ by its minimality (see Lemma \ref{finite}). In particular, this implies that the order of the monodromy group of $\E_\ebar$ is prime to $p$.

This article is devoted to determine how the assumptions of $X$ being projective over $S$ and of the order of the monodromy groups to be prime to $p$ can or cannot be relaxed in order to get similar results.
A first strengthening of the theorem comes rather directly from the ideas in the proof of Theorem \ref{el}. In order to prove it we need first to establish the following 

\begin{lemma}\label{dominate} 
Let $h:X\to S$ be a proper flat separable morphism of connected varieties  with geometrically connected fibers over an algebraically closed field $K$ and suppose it has a section $\sigma:S\to X$. Let $\tilde{S}\subset S(K)$ be any subset of the closed points of $S$, let $N\in\N$ and let us fix for every $s\in \tilde{S}$ a finite \'etale cover $g_s:Z_s\to X_s$ of degree less than $N$. Then there exists an open subscheme $\U\subset S$ and a finite \'etale cover $f:W\to X\times_S\U$ dominating all the $g_s$ for $s\in\tilde{S}\cap\U$; that is, for every $s\in\tilde{S}\cap\U$ the finite \'etale cover $f_s:W_s\to X_s$ factors through $g_s:Z_s\to X_s$.
\end{lemma}

\begin{proof}The proof of this lemma is a generalization of the construction that one can find in the beginning of the proof of \cite[Thm.~5.1]{EL:pcur}.

First notice that if the order of the $g_s:Z_s\to X_s$ is bounded by $N$ then the order of their Galois closures is bounded by $N!$, hence we can assume all the $g_s:Z_s\to X_s$ to be Galois. Moreover if $S'$ is connected and  $S'\to S$ is \'etale and generically finite then there is a non-trivial open $\U$ over which $S'\times_S\U\to \U$ is finite and \'etale. As $X\to S$ is smooth its image is open; hence, by shrinking $S$, we can assume that $X\to S$ is surjective.

Let $S'\to S$ be finite \'etale, then so is $X'=X\times_SS'\to X$. Let  $s\in \tilde{S}$ and $s'\in S'(K)$ lying over $s$. Assume that we have found $f':W\to X'$ such that $f'_s:W_s\to X_s'$ factors through $g_{s'}:Z_{s'}=Z_s\times_{k(s)}k(s')\to X'_{s'}$ then the composition $f:W\to X'\to X$ is a finite \'etale cover of $X$ and $f_s:W_s\to X_s$ factors through $g_s$.

As $h:X\to S$ has geometrically connected fibers so does $h':X'\to S'$ as if $s'\in S$ lies over $s\in S$ then $X'_{s'}=X_s\otimes_{k(s)}k(s')$. Therefore, $h':X'\to S'$ is proper, flat, separable and has geometrically connected fibers. So if $S'$ is connected the morphism $h':X'\to S'$ together with the section $\sigma':S'\to X'$ induced by $\sigma:S\to X$ satisfy the assumptions of the theorem.

\

For any $s\in \tilde{S}$ let $G_s\subset\pi_1^\text{\'et}(X_s,\sigma(s))$ be the open normal subgroup corresponding via Galois duality to the cover $g_s:Z_s\to X_s$.
Let $\ebar$ be a generic geometric point of $S$ given by the choice of an algebraic closure $\overline{k(S)}$ of $k(S)$. The fibers of $X\to S$ are geometrically connected and the morphism is proper, flat and separable; hence, the specialization map $\pi_1^\text{\'et}(X_\ebar,\sigma(\ebar))\onto \pi_1^\text{\'et}(X_s,\sigma(s))$ is surjective. Composing it with the quotient of $\pi_1^\text{\'et}(X_s,\sigma(s))$ by $G_s$ we get
\[\rho_s:\pi_1^\text{\'et}(X_\ebar,\sigma(\ebar))\onto \pi_1^\text{\'et}(X_s,\sigma(s))\onto \pi_1^\text{\'et}(X_s,\sigma(s))/G_s.\]
Notice that the index of $\ker(\rho_s)$ in $\pi_1^\text{\'et}(X_\ebar,\sigma(\ebar))$ is bounded by $N$. Let $\tau:S'\to S$ be any finite \'etale cover and let $s'\in S$ lying over $s$. As $K$ is algebraically closed, then $k(s')\simeq k(s)$ and hence $X'_{s'}\simeq X_s$. In particular, the natural morphism $\pi_1^\text{\'et}(X'_{s'},\sigma'(s'))\to\pi_1^\text{\'et}(X_s,\sigma(s))$ is an isomorphism. Let $G_{s'}\subset \pi_1^\text{\'et}(X'_{s'},\sigma'(s')) $ be the open subgroup corresponding to $g_{s'}:Z_{s'}\to X'_{s'}$, that is, the preimage of $G_s$ under this isomorphism and let us denote by
\[\rho_{s'}:\pi_1^\text{\'et}(X_\ebar,\sigma(\ebar))\onto \pi_1^\text{\'et}(X'_{s'},\sigma'(s'))\onto \pi_1^\text{\'et}(X'_{s'},\sigma'(s'))/G_{s'},\]

then $\ker(\rho_{s'})=\ker(\rho_s)\subset\pi_1^\text{\'et}(X_\ebar,\sigma(\ebar))$.

As $X_\ebar$ is a  projective $\overline{k(S)}$-variety, then $\pi_1^\text{\'et}(X_\ebar,\sigma(\ebar))$ is topologically finitely generated and hence has finitely many subgroups of index less than $N$, which are all opens by Nikolov--Segal theorem (\cite[Thm~1.1]{ns}), the intersection of which we denote by $G$: It is a normal open subgroup and it has finite index. Moreover 
\[ G\subset\bigcap_{s\in \tilde{S}}\ker(\rho_s)=\bigcap_{s'\in\tau^{-1}(\tilde{S})}\ker(\rho_{s'}).\]

At this point, we need an additional step before concluding similarly than in the proof of \cite[Thm.~5.1]{EL:pcur}. By Galois duality $G$ corresponds to a finite \'etale cover $Z_\ebar\to X_\ebar$. Let $k(S)^\text{sep}$ be the separable closure of $k(S)$ in $\overline{k(S)}$. The base change functor from the category of finite \'etale covers over $X\otimes_{S}k(S)^\text{sep}$ to the one of finite \'etale covers over $X_\ebar$ is an equivalence.  Hence,  $Z_\ebar$ is defined over some separable extension of $k(S)$. In particular, there exists an \'etale generically finite cover $S'\to S$ such that $Z_\ebar$ descends to a finite \'etale cover of $X'=X\times_SS'$.

 Let $\ebar'$ be the geometric generic point of $S'$ given by  $k(S)\subset k(S') \subset \overline{k(S)}$. Then $X'_{\ebar'}=X_\ebar$ and $\sigma(\ebar)=\sigma'(\ebar')$. Hence, the following diagram commutes:
\[\xymatrix{
\pi_1^\text{\'et}(X'_\ebar,\sigma'(\ebar'))\ar[r]\ar@{=}[d]& \pi_1^\text{\'et}(X',\sigma'(\ebar'))\ar[d]\\
\pi_1^\text{\'et}(X_\ebar,\sigma(\ebar))\ar[r]&\pi_1^\text{\'et}(X,\sigma(\ebar)).
}\]
Let $K'$ be the kernel of $\pi_1^\text{\'et}(X_\ebar,\sigma(\ebar))\to\pi_1^\text{\'et}(X',\sigma'(\ebar'))$. As $X\to S$ is projective we have the following exact sequence:
\[\xymatrix{&\pi_1^\text{\'et}(X_\ebar,\sigma(\ebar))\ar[d]^{q}\ar[dr]^\alpha\\
\{1\}\ar[r]&\pi_1^\text{\'et}(X_\ebar,\sigma(\ebar))/K'\ar[r]^i &\pi_1^\text{\'et}(X',\sigma'(\ebar'))\ar[r]&\pi_1^\text{\'et}(S',\ebar')\ar[r]\ar@/_.5pc/[l]_(.45){\sigma'_*}&\{1\}.
}\]

By \cite[V Cor 6.7]{SGA1} we have the inclusion $G\supset K'$. Moreover if we denote by $\Pi_{K'}=\pi_1^\text{\'et}(X_\ebar,\sigma(\ebar))/K'$, then the section $\sigma'_*$ induces a splitting
\[\pi_1^\text{\'et}(X',\sigma'(\ebar'))\simeq \Pi_{K'}\rtimes \sigma'_*(\pi_1^\text{\'et}(S',\ebar'))\]
as abstract groups. It is also a splitting of topological groups (see for example \cite[§2.10 Prop.~28]{Bour} and following discussion). In particular, the topology on $\pi_1^\text{\'et}(X',\sigma'(\ebar'))$ is the product topology.
Note that $\overline{G}=q(G)$ is invariant by the action of $\sigma'_*(\pi_1^\text{\'et}(S',\ebar'))$; hence, we can define
\[H=\overline{G}\rtimes \sigma'_*(\pi_1^\text{\'et}(S',\ebar')).\]
By definition $\alpha^{-1}(H)=G$, and $H$ has finite index in  $\pi_1^\text{\'et}(X',\sigma'(\ebar'))$. It is also open: $\pi_1^\text{\'et}(X',\sigma'(\ebar'))$ is endowed with the product topology, and $\overline{G}$ is open because $G=q^{-1}(\overline{G})$ is open as well. Hence, $H$ corresponds to a finite \'etale cover $W\to X'$ and as the composition $H\subset\pi_1^\text{\'et}(X',\sigma'(\ebar'))\onto\pi_1^\text{\'et}(S',\ebar')$ is surjective, then $W$ has geometrically connected fibers over $S'$. In particular, the specialization map is again surjective.
Let $z\in W$ be a closed point lying over $s'$ and let $\bar{\zeta}$ be a geometric generic point lying over $\sigma(\ebar)$, then we have the following commutative diagram

\[\xymatrix{\pi_1^\text{\'et}(W_\ebar,\bar{\zeta})\ar@/^1pc/[rr]^0\ar[r]\ar@{->>}[d]&\pi_1^\text{\'et}(X_\ebar,\sigma(\ebar))\ar[r]\ar@{->>}[d]&\pi_1^\text{\'et}(X_\ebar,\sigma(\ebar))/G\ar[d]
\\
\pi_1^\text{\'et}(W_s,z)\ar[r]&\pi_1^\text{\'et}(X'_{s'},\sigma'(s'))\ar[r]&\pi_1^\text{\'et}(X'_{s'},\sigma'(s'))/G_{s'}.
}
\]

Using the surjectivity of the specialization map on $W$ it follows that the composition of the morphisms on the second line is zero as well; hence, if $\tilde{G}_{s'}\subset \pi_1^\text{\'et}(X'_{s'},\sigma'(s'))$ is the open normal subgroup corresponding to $f_{s'}:W_{s'}\to X'_{s'}$ then $\tilde{G}_s\subset G_{s'}$. Therefore, $f_{s'}$ factors through $g_{s'}$.
\end{proof}

We can summarize the previous lemma by saying that with the assumptions of the theorem, up to shrinking $S$ every family of finite \'etale covers of the closed fibers with bounded order can be dominated by a finite \'etale cover of $X$ (notice that the existence of the section $\sigma:X\to S$ is not essential for the proof).

\begin{theorem}\label{eldiv}
Let $X\to S$ be a smooth \emph{proper}  morphism of $K$-varieties with geometrically connected fibers and $\E\in\Strat(X/S)$ of rank $r$. Assume that there exists a dense subset $\tilde{S}\subset S(K)$ such that, for every $s\in \tilde{S}$, the stratified bundle $\E_s$ has finite monodromy and that the highest power of $p$ dividing $|\pi(\E_s)|$ is bounded over $\tilde{S}$. Then
\begin{itemize}
\item[i)] there exists $f_\ebar:Y_\ebar\to X_\ebar$ a finite \'etale cover such that $f^*\E_\ebar$ decomposes as direct sum of stratified line bundles;
\item[ii)] if $K\neq \bar{\F}_p$ then $\E_\ebar$ is finite.
\end{itemize} 
\end{theorem}

\begin{proof}We will reduce this theorem to Theorem \ref{el}.
By the invariance of the monodromy group it suffices to prove the theorem for  $f^*\E$ where $f:Y\to X$ is a morphism of smooth $S$-varieties which is generically finite \'etale. By Chow's lemma and using de Jong alterations (\cite{dJ}) there exists $f:Y\to X$ projective and generically finite \'etale, hence we can assume $X$ to be projective.
Up to taking an \'etale open of $S$ we can assume that there exists a section $\sigma:S\to X$. For any $s\in \tilde{S}$ let $\Gamma_s=\pi(\E,\sigma(s))$ and $h_s:Y_s\to X_s$ the Picard--Vessiot torsor of $\E_s$ (see Lemma \ref{finite}). Let $G_s\subset \pi_1^\text{\'et}(X_s,\sigma(s))$ be the normal open subgroup corresponding via Galois duality to the cover $h_s$. By Tannakian duality $\E_s$ corresponds to the image of an $r$-dimensional representation of $\pi_1^{\Strat}(X_s,\sigma(s))$ (\cite[Prop.~2.21]{DM}) and as $\E_s$ is finite by \cite[Prop.~13]{dS:fun} this representation factors through the \'etale fundamental group, considered as a constant group scheme
\[\pi^{\Strat}_1(X_s,\sigma(s))\onto\pi^\text{\'et}_1(X_s,\sigma(s))\onto \pi^\text{\'et}_1(X_s,\sigma(s))\big/ G_s=\Gamma_s\subset GL_r(K)\]
where $r$ is the rank of $\E$. By Brauer--Feit generalization of Jordan's theorem \cite[Theorem]{bf}, as the orders of the Sylow-p-subgroups of every $G_s$ are bounded by $p^N$, there exists an integer $M=f(r,N)$ and, for every $s\in \tilde{S}$, a normal abelian subgroup $A_s$ such that $|\Gamma_s:A_s|<M$. This gives for every $s\in S(K)$ a Galois cover $g_s:Z_s\to X_s$ of order bounded by $M$ and a factorization 
\[ Y_s\to Z_s \to X_s\]
where $Y_s\to Z_s$ is Galois of Galois group $A_s$. Therefore, by Lemma \ref{dominate}, up to shrinking $S$ there exists a cover $g':Z'\to X$ such that $g'_s:Z'_s\to X_s$ factors through $g_s:Z_s\to X_s$. In particular, if $\E'$ is the pullback of $\E$ via $g'$ then $\pi(\E'_s)$ is abelian for every $s$.
Up to taking an \'etale open of $S$ the section $\sigma:S\to X$ extends to a section $\sigma':S\to Z'$. Let $\Gamma'_s=\pi(\E',\sigma'(s))$, as we just noticed for every $s\in \tilde{S}$ we have that $\Gamma'_s$ is abelian; hence, we can write it as the direct product of its $p$ part with its prime to $p$ part:
\[\Gamma'_s=\Gamma^p_s\times\Gamma^{p'}_s\]
and $\Gamma^{p'}_s$ corresponds to a Galois cover over $Z'_s$ whose index is by assumption bounded by $p^N$ for some $N\in\N$. Applying Lemma \ref{dominate} and up to shrinking $S$ we get a Galois cover $g'':Z''\to Z'$ dominating all such covers. Let $\E''$ be the pullback of $\E'$ along $g''$, then $\pi(\E''_s)$ is (abelian) of order prime to $p$ for every $s\in \tilde{S}$. Therefore, we have reduced the problem to Theorem \ref{el}.
\end{proof}

\section{A counterexample over countable fields}\label{esempio}

Our next aim is to drop the assumption of  $X$ being projective over $S$. However, before getting to the positive results, let us present a counterexample to understand what we can reasonably expect to hold without this assumption. Assume for the rest of this section $K$ to be an algebraically closed countable field. Let $X=\A^2_K$, $S=\A^1_K$ and let $X\to S$ be given by $K[y]\to K[x,y]$. The main result of this section is the following:
\begin{proposition}\label{ex}
There exists $\E\in\Strat(X/S)$ such that $\E_s$ is trivial for every point $s\in S(K)$ but $\E_\ebar$ is not isotrivial.
\end{proposition}

The rest of the section will be spent constructing such a stratified bundle and proving it satisfies the proposition.

As $x$ is a global coordinate of $X$ relative to $S$, it is clear that
\[\D_{X/S}=\Oh_X[\de_{x}^{(k)}\mid k\in\N_{>0}].\]
Moreover any vector bundle is free over $X$. If $E$ is a vector bundle on $X$, then a $\D_{X/S}$-module structure on $E$ is a $\Oh_S$-linear morphism 
\[\begin{split}
\phi:\D_{X/S}&\to \End_{\Oh_S}(E)
\end{split}
\]
extending the $\Oh_X$-module structure on $E$; hence, the image of $\Oh_X\subset\D_{X/S}$ under $\phi$ is always fixed. Therefore, to determine the action of the whole $\D_{X/S}$ it is enough to consider the image of the generators $\de_x^{(k)}$ under $\phi$.
 
Let $e_1,\dotsc,e_r$ be a basis for $E$, and let $A_k=(a_{ij}^k)$ be given by $\de_x^{(k)}(e_i)=\sum a_{ij}^ke_j$. Then the $A_k$, for $k\in\N_{>0}$, determine the $\D_{X/S}$-action: If $s=\sum_{i=1}^r f_i\cdot e_i$ is a section of $E$, with $f_i\in\Oh_X$, then using \eqref{product} we have
\begin{equation}\label{expli}\de_{x_i}^{(k)}(s)=\sum_{i=1}^r\sum_{\substack{a+b=k\\a,b\geq 0}}\de_x^{(a)}(f_i)A_b\cdot e_i.\end{equation}

Note that if $e_1',\dotsc,e_r'$ is an other basis of $E$ and $U=(u_{ij})\in\coh^0(X,\GL_r)$ is given by $e'_i=\sum u_{ij}e_j$ then by \eqref{product} it follows that in this new basis the matrices $A'_k=(a_{ij}^{'k})$ describing the action of $\de_x^{(k)}$ are given by
\begin{equation}\label{matrices}A'_k=\Big[\sum_{\substack {a+b=k\\a,b\geq 0 }}\de_x^{(a)}(U)A_b\Big]U^{-1}.\end{equation}

\

In order to construct our example, let us fix a bijection $n\mapsto a_n$ between the natural numbers and $K=S(K)$. Let $\E\in\Strat(X/S)$ be the rank-two relative stratified bundle $E=\Oh_X\cdot e_1\oplus\Oh_x\cdot e_2$ with $\D_{X/S}$-action given by $\de_x^{(k)}(e_1)=0$ and

\begin{equation}\label{exa}\de_x^{(k)}(e_2)=\begin{cases}
 \prod_{i=0}^h (y-a_i) \cdot e_1
 &\text{if } k=p^h, \\
\quad0&\text{else.}
\end{cases} \end{equation}
In order to prove Proposition \ref{ex} we need to show that this actually defines an action of $\D_{X/S}$ over $E$ and that $\E$ satisfies the two properties of the proposition, namely that it is trivial on every closed fiber and not isotrivial on the geometric generic fiber.

\begin{lemma}
The formulae in \eqref{exa} define a $\D_{X/S}$-module structure on $E$.
\end{lemma}
\begin{proof}
As we fixed the action of the generators of $\D_{X/S}$, for it to extend to a $\D_{X/S}$-action we only need to check that the relations of the generators in the ring of differential operators are satisfied by their images in $\End_{\Oh_X}(E)$.
By \cite[Cor.~2.5]{Ba:gen} the only relations are
\[[\de_x^{(l)},\de_x^{(k)}]=0\]
\[\de_x^{(k)}\circ\de_x^{(l)}=\bin{k+l}{k}\de_x^{(k+l)}\]
\[[\de_x^{(k)},x]=\de_x^{(k-1)}.\]
Let us begin with the second relation: for $k,l>0$
\[\de_x^{(k)}\circ\de_x^{(l)}(e_1)=0\]
\[\de_x^{(k)}\circ\de_x^{(l)}(e_2)=\left\{\begin{array}{cl}\de_x^{(k)}(\prod_{i=0}^h (y-a_i) \cdot e_1)=0
 &\text{ if } l=p^h \\
0&\text{ else}\end{array}\right.\]
Hence, we just need to verify that if $k+l=p^h$ then $\bin{k+l}{k}=0$ but this holds by Lucas's theorem and the first relation follows immediately. Moreover by \eqref{expli} we have
\[\de_x^{(k)}\cdot x(e_i)=\de_x^{(k)}(x\cdot e_i)=\sum_{\substack{a+b=k\\a,b\geq 0}}\de_x^{(a)}(x)\de_x^{(b)}(e_i)=x\de_x^{(k)}(e_i)+\de_x^{(k-1)}(e_i);\]
hence, the third relation trivially holds.
\end{proof}

In order to prove that $\E_s$ is trivial for every closed fiber, let us fix $n\in\N$ and let $s=a_n\in S(K)$, that is, $X_s=\{y=a_n\}\subset X$. Let us consider the basis change on $\Oh_{X_s}\cdot e_1\oplus\Oh_{X_s}\cdot e_2=E_{s}$
given by $e'_1=e_1$ and
\[ e'_2=e_2-\Big[(y-a_0)x+(y-a_0)(y-a_1)x^p+\dotsb+\big[\prod_{i=0}^{n-1}(y-a_i)\big]x^{p^{n-1}}\Big]\cdot e_1\]
then by \eqref{matrices} in this new basis the action of $\D_{X_s/k(s)}$ is given by $\de_x^{(k)}(e'_1)=\de_x^{(k)}(e'_2)=0$ hence is the trivial action. 

We are now left to prove that $\E_\ebar$ is not isotrivial:

\begin{lemma}
Let $\E\in\Strat(X/S)$ be the stratified bundle defined by \eqref{exa}, then $\E_\ebar$ is not isotrivial.
\end{lemma}
\begin{proof}
In order to prove that $\E_\ebar$ is not isotrivial it suffices by Lemma \ref{finito} to show that it cannot be defined over any $K'$ of finite type over $\F_p$. Remark that $\E_\ebar$ is defined over $\A^1_{\ebar}$ which descends to $\A^1_{K'}$ for any $K'\subset K$. By way of contradiction assume there exists such a $K'$ and let $\E'$ be the descent of $\E_\ebar$ over $\A^1_{K'}$. This means that there is a basis $e'_1,e'_2$ of $\E_\ebar$ such that the matrices $A'_k$ in this new basis take values in $K'[x]$. 

Let $U$ be the basis change matrix between $e_i$ and $e'_i$, then $U$ is defined over some $K''$ of finite type over $K'$, hence over $\F_p$, so by \eqref{matrices} we have that $\prod_{i=0}^h (y-a_i) \in K''[x]$. In particular, if we denote by $\mathcal{A}=\F_p[\prod_{i=0}^h (y-a_i)\mid h\in\N]$, our assumption implies that $\mathcal{A}\subset K''[x]$. 

To see that this leads to a contradiction it suffices to show that $\mathcal{K}\nsubseteq K''(x)$ where $\mathcal{K}$ is the quotient field of $\mathcal{A}$. Note that $K\subset\mathcal{K}$; therefore, it is enough to prove that for every $K'$ of finite type over $\F_p$  we have that $K\nsubseteq K'(x)$ and as $\bar{\F}_p\subset K$ it is sufficient to show $\bar{\F}_p\nsubseteq K'(x)$, which follows from the following:
\begin{claim}Let $\F_q$ be a finite field with $q=p^n$ for some $n\in\N$ and let $K\supset\F_q$ an algebraic extension such that $[K:\F_q]=+\infty$. Then for every $m\in\N$ and every $\varepsilon_1,\dotsc,\varepsilon_m$ non-algebraic over $\F_q$ we have that
\[K\nsubseteq \F_q(\varepsilon_1,\dotsc,\varepsilon_m).\]
\end{claim}
\begin{proof}
By induction on $m$, the case $m=0$ being evident. 
Let $m=1$, and  $\gamma\in K-\F_q$ and let $\mu_\gamma(t)$ its minimal polynomial over $\F_q$. By way of contradiction assume $\gamma\in \F_q(\varepsilon_1)$; then $\gamma=f(\varepsilon_1)/g(\varepsilon_1)$ and 
\[g(\varepsilon_1)^{\deg\mu_\gamma}\cdot\mu_\gamma\big(\frac{f(\varepsilon_1)}{g(\varepsilon_1)}\big)=0\]
 gives an algebraic dependence of $\varepsilon_1$ over $\F_q$ which is a contradiction with our assumption that  $\varepsilon_1$ is not algebraic over $\F_q$.
Let now $m\geq 1$, by induction step we know that for every $n\in\N$, $q=p^n$, then no infinite algebraic extension of $\F_q$ is contained in $\F_q(\varepsilon_1,\dotsc,\varepsilon_{m-1})$; hence, there exists an $r$ such that $\F_{q^r}=\F_q(\varepsilon_1,\dotsc,\varepsilon_{m-1})\cap K$. Then 
\[\F_q(\varepsilon_1,\dotsc,\varepsilon_{m-1})(\varepsilon_m)\cap K\subset\F_{q^r}(\varepsilon_m)\neq K\]
 by the $m=1$ step applied to $q=p^{rn}$. In particular, $K\nsubseteq \F_q(\varepsilon_1,\dotsc,\varepsilon_m)$.
\end{proof}
Note that if $K'$ is of finite type over $\F_p$ then $K'$ can be always be written as $\F_q(\varepsilon_1,\dotsc,\varepsilon_m)$ for some $q=p^n$ and $\varepsilon_i$ non-algebraic over $\F_q$; hence, $\bar{\F}_p\nsubseteq K'(x)$. Therefore, $\E$ cannot be defined over any $K'$ of finite type over $\F_p$ and by Lemma \ref{finito} it cannot be finite.
\end{proof}

\begin{remark} 
Let us observe that if $K$ is uncountable the same construction provides an example of a relative stratified bundle $\E\in\Strat(\A^2_K/\A^1_K)$ and a dense subset $\tilde{S}\subset \A^1_K(K)$ such that $\E_s$ is trivial for every $s\in \tilde{S}$ but $\E_\ebar$ is not isotrivial. Therefore, the density condition on $\tilde{S}$ of Theorem \ref{el} will not be sufficient for our purposes, in parallel with the similar problem that one encounters in the equicharacteristic zero case (see \cite[Rmk.~7.2.3]{An:pcur}).
\end{remark}

\section{The main theorem}
From the example in previous section it appears that in the case where $X$ is not projective over $S$ the situation is significantly different from the one in Theorem \ref{el}. In the latter one big obstruction for the theorem to hold is related to $p$ dividing the order of the monodromy group on the closed fibers. In the counterexample of Section \ref{esempio} these are trivial and the obstruction seems more related to the cardinality of $K$. As noticed in Section \ref{monodromy}, the monodromy group does not depend (up to a non-unique isomorphism) on the choice of $x\in X$. Therefore, in this section we will denote the monodromy group of a stratified bundle $\E$ simply by $\pi(\E)$.

In order to phrase the statement of the main theorem let us introduce the following notation: we will denote by $(X,S;\E)$ (and call it \emph{a triple over $K$}) any triple consisting of $X\to S$ smooth morphism of $K$-varieties with geometrically connected fibers and $\E\in\Strat(X/S)$. We  denote furthermore by $K'=K'(X,S;\E)$ a (minimal) algebraically closed subfield of $K$ such that $(X,S;\E)$ is defined over $K'$, and by $(X',S';\E')$ the descent of the triple $(X,S;\E)$ to $K'$. Then the following result holds

\begin{theorem}\label{main}
Let $(X,S;\E)$ be a triple over $K$, let $K'=K'(X,S;\E)\subset K$ and $(X',S';\E')$ the descended  triple to $K'$. Let $k(S')$ be the function field of $S'$. Let us assume:
\[\exists \quad i:k(S')\hookrightarrow K \text{ extending }K'\subset K. \qquad (*)\]
Assume that there exists a dense open $\tilde{S}\subset S$ such that \(\E_s\) is finite for every $s\in \tilde{S}(K)$, then so is $\E_\ebar$. Moreover there exists $s\in \tilde{S}(K)$ such that $\pi(\E_s)(K)\simeq \pi(\E_\ebar)(\overline{k(S)})$, in particular
\begin{itemize}
\item[i)] $|\pi(\E_s)|$ is bounded over $\tilde{S}(K)$;
\item[ii)]if $p\nmid|\pi(\E_s)|$ for every $s\in \tilde{S}(K)$ then $p\nmid|\pi(\E_\ebar)|$;
\item[iii)]any group property holding for $\pi(\E_s)$ for every $s\in \tilde{S}(K)$ holds for $\pi(\E_\ebar)$.
\end{itemize}
\end{theorem}
\begin{proof} Up to shrinking $S$, we can assume $\tilde{S}=S$.
Let $\Delta:S'\to S'\times_{\Spec K'} S'$ be the diagonal morphism and $i:k(S')\hookrightarrow K$ be an immersion as in $(*)$. Then the base change along 
\[
\xymatrix{\Spec K\ar[r]^i&\Spec k(S')\ar[r]&\Spec S'
}\]
induces  $(s:\Spec K\to S)\in S(K)$ such that $X'\otimes_{k(S')}K\simeq X_s$ and 
\[i^*\E'=\E_s,\]
where we are considering $i:k(S')\hookrightarrow K$ as a geometric generic point of $S'$.
In particular, $\pi(i^*\E')=\pi(\E_s)$. Therefore, $i^*\E'$ is finite; hence, by Lemma \ref{basechange}, $\E_\ebar$ is finite as well, moreover their monodromy group are isomorphic as constant groups.
\end{proof}

\begin{corollary}\label{maincor}
If $K$ is uncountable then the assumption $(*)$, hence the theorem, always holds.
\end{corollary}
\begin{proof}
It suffices to show that if  $K$ is uncountable, then for any triple $(X,S;\E)$ over $K$ there exists  $K'=K'(X,S;\E)$ and  an inclusion $k(S')\subset K$ extending $K'\subset K$. But it is easy to check that a triple $(X,S;\E)$ is defined by countably many data; hence, we can choose $K'$ such that it has countable transcendence degree over $\F_p$. As $K$ is uncountable, it has infinite transcendence degree over $K'$; hence, there always exists $k(S')\subset K$ as in $(*)$.
\end{proof}

\begin{remark}Notice that if the smooth morphism $X\to S$ does not have geometrically connected fibers then we lose the notion of monodromy group on the closed and geometric generic fibers: if $X$ is a $K$-variety which is not connected and $\mathbb{I}_{X/K}$ is the trivial stratified bundle on $X$ then $\End (\mathbb{I}_{X/K})\neq K$, hence $\Strat(X/K)$ is not a Tannakian category. Nevertheless, if we do not assume $X\to S$ to have geometrically connected fibers, the same proof shows that if $\E_s$ is finite when restricted to every connected component of $X_s$, then the same holds for $\E_\ebar$ on every connected component of $X_\ebar$.
\end{remark}

\section{Regular singularity and a refinement of the theorem}
Regardless of the example in Section \ref{esempio}, there is a way to broaden Theorem \ref{el} in the case where $K$ is countable, making the additional assumption that the stratified bundle is regular singular on the geometric generic fiber.

Let $X$ be a smooth variety over $K$ and let $(X,\overline{X})$ be a \emph{good partial compactification} of $X$; that is: $\overline{X}$ is a smooth variety over $K$ such that $X\subset\overline{X}$ is an open subscheme and $D=\overline{X}\backslash X$ is a strict normal crossing divisor. Let $\D_{\overline{X}/K}(\log D)\subset \D_{\overline{X}/K}$ the subalgebra generated by the differential operators that locally fix all powers of the ideal of definition of $D$. If $\U\subset \overline{X}$ admits global  coordinates $x_1,\dotsc,x_d$ and $D$ is smooth and given by $\{x_1=0\}$ then 
\[\D_{\overline{X}/K}(\log D)_{|\U}=\Oh_\U\big[x_1^k\de_{x_1}^{(k)},\de_{x_i}^{(k)}\mid i\in\{2,\dotsc,d\},k\in\N_{>0}\big].\]
\begin{definition}
A stratified bundle $\E\in\Strat(X/K)$ is called \emph{$(X,\overline{X})$-regular singular} if it extends to a locally free \(\Oh_{\overline{X}}\)-coherent $\D_{\overline{X}/k}(\log D)$-module $\overline{\E}$ on $\overline{X}$. It is \emph{regular singular} if it is $(X,\overline{X})$-regular singular for every partial good compactification $(X,\overline{X})$.
\end{definition}

\begin{remark}
There is a parallel notion of regular singularities in characteristic zero. Despite the fact that isotrivial implies regular singular over the complex numbers, this is not longer true in positive characteristic, due to the existence of wild coverings (for a more precise statement, see \cite[Thm.~1.1]{Kin:rs}).
\end{remark}
For a $(X,\overline{X})$-regular singular stratified bundle $\E$ we have a theory of exponents (see \cite[§3]{Gie:flat}) of $\E$ along $D$: it is a finite subset $\Exp_D(\E)\subset\Z_p/\Z$ given by the following:

\begin{proposition}\emph{\cite[Lemma~3.8]{Gie:flat},\cite[Prop.~4.12]{Kin:rs} }\label{exp}Let $\overline{X}=\Spec A$ be a smooth variety over $K=\bar{K}$ with global coordinates $x_1,\dotsc,x_d$ and let $D$ be the smooth divisor defined by $\{x_1=0\}$. Let $\E\in\Strat(X/K)$ a $(X,\overline{X})$-regular singular stratified bundle and $\overline{\E}$ a locally free $\D_{\overline{X}/K}(\log D)$-module extending $\E$ . Then there exists a decomposition of $\overline{\E}_{|D}=\bigoplus K_\alpha$ with $\alpha\in\Z_p$ such that $x_1^k\de_{x_1}^{(k)}$ acts on $K_\alpha$ by multiplication by ${\alpha\choose k}$. The image in $\Z_p/\Z$ of the $\alpha\in\Z_p$ such that $K_\alpha\neq 0$  are called the \emph{exponents} of $\E$ along $D$ and do not depend on the choice of $\overline{\E}$.
\end{proposition}

If $D$ is not smooth $\Exp_D(\E)$ is defined to be the union of the exponents along all the irreducible components of $D$. By \cite[Cor.~5.4]{Kin:rs} $\E$ extends to a stratified bundle $\overline{\E}$ on $\overline{X}$ if and only if its exponents are zero. In particular, \cite[Prop.~4.11]{Kin:rs} implies that if $\E$ is finite then its exponents are torsion. Moreover:

\begin{lemma}\label{torsion}
Let $\E$ be a $\D_{X/S}$-module such that $\E_s$ is finite for every $s\in\tilde{S}$ a dense subset of $S(K)$. If $\E_\ebar$ is regular singular then the exponents of $\E_\ebar$ with respect to any  partial good compactification of $X_{\bar{\eta}}$ are torsion.
\end{lemma}
\begin{proof}
Let us fix $(X_\ebar,\overline{X}_\ebar)$  a partial good compactification and let $D_\ebar=\overline{X}_\ebar\backslash X_\ebar$. As the exponents can be checked locally, we can shrink $\overline{X}_{\bar{\eta}}$ around the generic point of one of the irreducible components of $D_{\bar{\eta}}$ at a time. Moreover in order to prove the lemma we are allowed to take a generically finite \'etale open $S'$ of $S$ and substitute $X$ by $X\times_S S'$ (and $\tilde{S}$ by its preimage) as the geometric generic fiber is the same. Finally for every $s\in S(K)$ we have that $X'_s$ is either empty  or a finite union of copies of $X_s$; hence, we will still denote by $s$ any point $s'\in S'$ lying over it.

Hence, without loss of generality, we can assume that we are in the following situation: the partial good compactification $(X_\ebar, \overline{X}_\ebar)$ is the restriction of a relative good partial compactification $(X,\overline{X})$ defined on the whole $S$, $\overline{X}$ is the spectrum of a ring $A$,  with global relative coordinates $x_1,\dotsc,x_d$ over $S$ and finally  $D=\overline{X}\backslash X$ is defined by $\{x_1=0\}$. Moreover we can assume that $\E$ is globally free and that on the geometric generic fiber $\E_\ebar$ extends to a globally free \(\D_{X_{\bar{\eta}}/\overline{k(S})}\)-module $\overline{\E}_\ebar$. 

Let $s\in \tilde{S}$ be any point such that $X_s\cap D\neq\emptyset$, and let us consider the globally free $\Oh_X$-module $\bar{E}=\Oh_{\overline{X}}\bar{e}_1\oplus\dotsb\oplus\Oh_{\overline{X}}\bar{e}_r$.  Then the $\bar{e}_i$ induce a basis on the restriction of $\overline{E}$ to the closed fiber over $s$ (as well as to the geometric generic one) and to the boundary divisor (as well as to its complement) as in the following commutative diagram:
\[\xymatrix{
E_s=\bigoplus_{i=1}^r\Oh_{X_s}e^s_i&\ar[l]_{\otimes k(s)}E=\bigoplus_{i=1}^r\Oh_Xe_i\ar[r]^{\otimes\overline{k(S)}}&E_\ebar=\bigoplus_{i=1}^r\Oh_{X_\ebar}\varepsilon_i
\\
\overline{E}_s=\bigoplus_{i=1}^r\Oh_{\overline{X}_s}\bar{e}^s_i\ar[u]^{|X_s}\ar[d]_{|_{D_s}}&\ar[l]_{\otimes k(s)}\overline{E}=\bigoplus_{i=1}^r\Oh_{\overline{X}}\bar{e}_i\ar[r]^{\otimes\overline{k(S)}}\ar[u]^{|_X}\ar[d]_{|_D}&\overline{E}_\ebar=\bigoplus_{i=1}^r\Oh_{\overline{X}_{\bar{\eta}}}\bar{\varepsilon}_i\ar[u]^{|_{X_{\bar{\eta}}}}\ar[d]_{|_{D_{\bar{\eta}}}}
\\
\overline{E}_{|D_s}=\bigoplus_{i=1}^r\Oh_{D_s}\tilde{e}^s_i&\ar[l]_{\otimes k(s)}\overline{E}_{|D}=\bigoplus_{i=1}^r\Oh_D\tilde{e}_i\ar[r]^{\otimes\overline{k(S)}}&\overline{E}_{|D_\ebar}=\bigoplus_{i=1}^r\Oh_{D_{\bar{\eta}}}\tilde{\varepsilon}_i .
}
\]
Consider the first line of the diagram: on the first (respectively second and third) column there is an action of $\D_{X_s/k(s)}$ (respectively $\D_{X/S}$ and $\D_{X_{\bar{\eta}}/\overline{k(S)}}$), compatible with each other. On the last column this action extends to a logarithmic action on $\overline{E}_\ebar$ that we want to extend compatibly to $\overline{E}$. 

Similarly as in Section \ref{esempio}, let $A_{i,k}$ be the matrices describing the action of $\de_{x_i}^{(k)}\in\D_{X/S}$ in the basis $e_i$, then the same ones describe the action of $\de_{x_i}^{(k)}\in\D_{X_\ebar/\overline{k(S)}}$ in the basis $\varepsilon_i$. By regular singularity of $\E_\ebar$ this action extends to a $\D_{\overline{X}_\ebar/\overline{k(S)}}(\log D_\ebar)$-action. Therefore, there is a second basis  $\varepsilon'_1,\dotsc,\varepsilon'_d$ on the geometric generic fiber such that in the new basis the matrices $A'_{i,k}$ have no poles in $x_1$ for $i\neq 1$ and logarithmic poles for $i=1$. Let $U\in\coh^0(X_\ebar,\GL_r)$ the basis change matrix from $\varepsilon_i$ to $\varepsilon'_i$. Taking a generically finite \'etale open of $S$ we can assume that $U$ is defined on the whole $S$; hence, the $A'_{i,k}$ are defined over the whole $S$ as well and this defines an action of 
\[\D_{X/S}(\log D)\doteq \Oh_X\big[x_1^k\de_{x_1}^{(k)},\de_{x_i}^{(k)}\mid i\in\{2,\dotsc,d\},k\in\N_{>0}]\]
on $\overline{E}$, compatible with the logarithmic action on the fibers over $\ebar$. In particular, this induces a $\D_{\overline{X}_s}(\log D_s)$-action on $\overline{E}_s$; hence, $\E_s$ is $(X_s,\overline{X}_s)$-regular singular (notice that if $S'$ is an \'etale open of $S$ then for $s\in S(K)$ the fiber $X'_s$ of $X'=X\times_S S'$ is either empty or the disjoint union of finitely many copies of $X_s$).

We want now to compare $\Exp_{D_\ebar}(\E_\ebar)$ and  $\Exp_{D_s}(\E_s)$. By Proposition \ref{exp} we have that $\overline{\E}_{\ebar|_{D_\ebar}}=\oplus F_\alpha$; hence, there exists $\tilde{\varepsilon}_i$ a basis of $\overline{E}_{|D_\ebar}$ such that the matrices $\tilde{B}_k$ defining the action of $x_1^k\de_{(k),x_1}$ are diagonal with values $\bin{\alpha}{k}\in\F_p$. Let $\bar{\varepsilon}_i$ be a lift of $\tilde{\varepsilon}_i$, then up to taking an \'etale generically finite open of $S$ we can assume that $\bar{\varepsilon}_i$ is a restriction of a basis $\bar{e}_i$ of $\overline{E}$ over $\overline{X}$. In particular, the decomposition extends as well and $\overline{\E}_{|D}=\oplus F_\alpha$ induces a decomposition on $\overline{\E}_{s|_{D_s}}$. This decomposition must coincide with the one given by Proposition \ref{exp}; hence, the exponents must be the same of the ones of $\E_\ebar$. As $\E_s$ is isotrivial, its exponents are torsion; hence, so must be the ones of $\E_\ebar$.
\end{proof}

\begin{remark}
While the previous proof shows that if $\E_\ebar$ is regular singular so are the $\E_s$ for every $s\in S(K)$, the example in Section \ref{esempio}, together with Theorem~\ref{smoothrs}, shows that the converse does not hold in general (however, one can prove it is the case when $K$ is uncountable). On the contrary, in characteristic zero it is always true that if a relative flat connection is regular with respect to some smooth good compactification on the fibers over a dense set of points of $S$, then it is regular on the geometric generic fiber, as proven in  \cite[Lemma~8.1.1]{An:pcur}.
\end{remark}

Before stating and proving the main theorem of this section we need to prove the existence of Kawamata coverings in positive characteristic. Analogously to the original construction in characteristic zero (\cite[Thm.~17]{kaw}) we have the following
\begin{theorem}\label{kawamata}
Let $X$ be a projective smooth variety  of dimension $d$ over an algebraically closed field $K$ of characteristic $p$ and let $D$ be a simple normal crossing divisor on $X$. Let $m\in\N$ prime to $p$, then there exist a projective smooth variety $Y$ and a finite surjective mapping $f:Y\to X$ such that $(f^*D)_{red}$ is a simple normal crossing divisor on $Y$ and if $f^*D=\sum m_i \tilde{D}_i$ is the decomposition in irreducible components with $\tilde{D}_i\neq\tilde{D}_j$ for $i\neq j$ then $m\mid m_i$ for all $i$ and $m_i$ are all prime to $p$.
\end{theorem}

\begin{proof} The proof follows the one of the original theorem (\cite[Thm.~17]{kaw}). One does the construction one irreducible component $D'$ of $D$ at a time, choosing an ample line bundle $\M$ on $X$ and $N\gg0$ such that $N\M-D'$ is very ample. The only additional care that needs to be taken, is to choose $N$ so that $m\mid N$ and $(N,p)=1$, which is possible as $m$ is prime to $p$ (this will be enough to prove that $Y$ is smooth in the very same way by \cite[Lemma~1.8.6]{Gro}). One needs moreover to use  \cite[Cor.~12]{klei} instead of the classical smoothness theorem for general members of a very ample linear system.
\end{proof}

We can now state and prove the following

\begin{theorem}\label{smoothrs}
Let $X\to S$ be a smooth morphism of $K$-varieties with geometrical connected fibers and let $\E\in\Strat(X/S)$. Assume that there exists a dense subset  $\tilde{S}\subset S(K)$ such that, for every $s\in\tilde{S}$, the stratified bundle $\E_s$ has finite monodromy and that the highest power of $p$ dividing $|\pi(\E_s)|$ is bounded over $\tilde{S}$. Assume moreover that $\E_\ebar$ is regular singular, then
\begin{itemize}
\item[i)] there exists $f_\ebar:Y_\ebar\to X_\ebar$ a finite \'etale cover such that $f^*\E_\ebar$ decomposes as direct sum of stratified line bundles;
\item[ii)] if $K\neq \bar{\F}_p$ then $\E_\ebar$ is finite.
\end{itemize} 
\end{theorem}
\begin{proof}
Let $\U\subset X$ be a dense open, then by invariance of the monodromy group it is enough to show the theorem for $\E_{|\U}$ moreover it is enough to prove finiteness for its pullback along any finite \'etale cover. Therefore, we can always work up to generically finite \'etale covers.
Using \cite{dJ} we can find an alteration generically finite \'etale $f:X'\to X$ such that $X'$ admits a good projective compactification relative to $S$.  By \cite[Prop.~4.4]{Kin:rs} the pullback of a regular singular stratified bundle is again regular singular. Hence, without loss of generality, we can assume that $X$ admits a good projective compactification $\overline{X}$ relative to $S$. We will denote by $D=\overline{X}\backslash X$ the divisor at infinity.

Let $\Exp_D(\E_\ebar)\subset \Z_p/\Z$ be the finite set of exponents of $E_\ebar$ along $D_\ebar$ (as defined in Lemma \ref{exp}). As $\E_\ebar$ is regular singular then  by Lemma \ref{torsion} the exponents of $\E_\ebar$ are torsion; let $m\in\N$ an integer prime to $p$ killing the torsion of $\Exp_D(\E_\ebar)$ and let $f:\overline{Y}_\ebar\to \overline{X}_\ebar$ be the Kawamata covering constructed in Theorem \ref{kawamata}: it ramifies on a simple normal crossing divisor $\tilde{D}_\ebar$ containing the divisor at infinity $D_\ebar=\overline{X}_\ebar-X_\ebar$ and it is Kummer on $\overline{X}_\ebar-\tilde{D}_\ebar$. As $m$ divides the ramification order along $D_\ebar$ by \cite[Prop.~4.11]{Kin:rs} the exponents of the pullback of $\E_\ebar$ along $(f^*D_\ebar)_{red}$ are zero; hence, it extends to the whole $Y_\ebar$. Up to taking an \'etale open of $S$ and using a similar argument as in the proof of Lemma \ref{torsion} we can assume that this extension is defined on the whole $S$. Therefore, we have reduced the problem to Theorem \ref{eldiv}.
\end{proof}

\section{Finite vector bundles}

The notion of isotriviality has as well relevance in the category of vector bundles over a proper smooth $K$-variety, even though it is not equivalent to the notion of finiteness (see \cite[Lemma 3.1]{no} and following definition) for vector bundles, at least in positive characteristic. In the same spirit of Theorem \ref{el}, Esnault and Langer proved in the same paper the following:

\begin{theorem}\emph{\cite[Thm.~5.1]{EL:pcur}}\label{elbu} Let $X\to S$ be a  smooth projective morphism of $K$-varieties with geometrically connected fibers and let $E$ be a locally free sheaf over $X$. Assume that there exists a dense subset $\tilde{S}\subset S(K)$ such that, for every $s\in \tilde{S}$, there is a finite \'etale Galois cover  $h_s:Y_s\to X_s$  \emph{of order prime to $p$} such that $h_s^*(E_s)$ is trivial. Then
\begin{itemize}
\item[i)] there exists $f_\ebar:Y_\ebar\to X_\ebar$ a finite \'etale cover \emph{of order prime to $p$} such that $f^*E_\ebar$ decomposes as direct sum of stratified line bundles;
\item[ii)] if $K\neq \bar{\F}_p$ then $E_\ebar$ is trivialized by a finite \'etale cover \emph{of order prime to $p$}.
\end{itemize} 
\end{theorem}

Then, a reasoning similar to the proof of Theorem~\ref{eldiv} proves the following:
\begin{theorem}Let $X\to S$ be a  smooth \emph{projective} morphism of $K$-varieties with geometrically connected fibers and let $E$ be a locally free sheaf over $X$. Assume that there exists a dense subset $\tilde{S}\subset S(K)$ such that, for every $s\in \tilde{S}$, there is a finite \'etale Galois cover $h_s:Y_s\to X_s$ such that $h_s^*(E_s)$ is trivial and that the highest power of $p$ dividing the order of such covers is bounded over $\tilde{S}$. Then
\begin{itemize}
\item[i)] there exists $f_\ebar:Y_\ebar\to X_\ebar$ a finite \'etale cover  such that $f^*E_\ebar$ decomposes as direct sum of stratified line bundles;
\item[ii)] if $K\neq \bar{\F}_p$ then $E_\ebar$ is trivialized by a finite \'etale cover.
\end{itemize} 
\end{theorem}
\begin{proof}
We will reduce this theorem to Theorem \ref{elbu}. By taking an \'etale open of $S$ we can assume there exists a section $\sigma:S\to X$. Let $r$ be the rank of $E$ and fix $s$ a closed point in $S$. As $X_s$ is a smooth $k(s)$-variety, then (see \cite[Definition 3.2]{EL:pcur} and following discussion) every \'etale trivializable vector bundle is Nori semistable. In particular, the Galois cover $h_s:Y_s\to X_s$ corresponds to a representation of rank $r$ of the Nori fundamental group scheme $\pi^N_1(X_s,\sigma_s)$ (for the definition of the Nori group scheme see \cite{no}) that factors through the \'etale fundamental group:
\[\pi_1^N(X_s,\sigma(s))\onto \pi_1^\text{\'et}(X_s,\sigma(s))\onto \Gamma_s\subset \GL_r(K),\]
where $\Gamma_s$ is the Galois group of $h_s:Y_s\to X_s$.
The rest of the proof follows exactly as in Theorem \ref{eldiv}.
\end{proof}

If the morphism $X\to S$ is not projective but only smooth we get a similar result to Corollary~\ref{maincor}:

\begin{theorem}
Let $K$ be an algebraically closed field of positive characteristic with infinite transcendental degree over $\F_p$. Let $X\to S$ be a smooth morphism of varieties over $K$ and $E$ a vector bundle over $X$.
Assume that there exists a dense open $\tilde{S}\subset S$ such that \(E_s\) is isotrivial for every $s\in \tilde{S}(K)$, then so is $E_\ebar$.
\end{theorem}
\begin{proof} There exists $K'$ a subfield of $K$ of finite type over $\F_p$ such that $X\to S$ and $E$ descend to $X'\to S'$ and $E'$. Moreover as $K$ has infinite transcendence degree over $\F_p$ there exists an immersion $k(S')\hookrightarrow K$ over $K'$ and a point $s\in S(K)$, like in the proof of Theorem~\ref{main}, such that the morphism $i:\Spec K\to \Spec k(S')$  given by $k(S')\subset K$ is a geometric generic point of $S'$, and on $X'\otimes_{k(S')}K\simeq X_s$ 
\[i^*E'=E_s.\]
Note that there exists an immersion $\iota:K\hookrightarrow \overline{k(S)}$ (which is not the natural one given by the fact that $S$ is a $K$-variety) that is the identity on $k(S')$, hence via $\iota$ we have that $X_\ebar\simeq X'\otimes_{k(S')}\overline{k(S)}\simeq X_s\otimes_K\overline{k(S)}$. In particular, if we continue to consider $K$ as a subfield of $\overline{k(S)}$ via the immersion $\iota$, then  $h_s\otimes_K\overline{k(S)}:Y_s\times_{\Spec K}\Spec \overline{k(S)}\to X_\ebar$ trivializes $E_\ebar$.
\end{proof}

\addcontentsline{toc}{section}{\refname}
\printbibliography
\end{document}